\documentclass[12pt]{amsart}

\usepackage{amsmath}
\usepackage{amssymb}
\usepackage{amsthm}
\usepackage{amscd}
\newtheorem{thm}{Theorem}[section]
\newtheorem{lem}[thm]{Lemma}
\newtheorem{lem-dfn}[thm]{Lemma-Definition}
\newtheorem{prop}[thm]{Proposition}
\newtheorem{cor}[thm]{Corollary}

\theoremstyle{definition}
\newtheorem{defn}[thm]{Definition}
\newtheorem{exam}[thm]{Example}
\newtheorem{ex}[thm]{Example}

\newtheorem{quest}[thm]{Question}

\newtheorem{prob}[thm]{Problem}

\newtheorem*{acknowledgement}{Acknowledgement}
\theoremstyle{remark}

\newtheorem{rem}[thm]{Remark}
\numberwithin{equation}{section}

\newcommand{\thmref}[1]{Theorem~\ref{#1}}
\newcommand{\lemref}[1]{Lemma~\ref{#1}}

\newcommand{\proref}[1]{Proposition~\ref{#1}}

\DeclareMathOperator{\Spec}{Spec}
\DeclareMathOperator{\spec}{Spec}

\DeclareMathOperator{\supp}{Supp}

\DeclareMathOperator{\di}{div}

\DeclareMathOperator{\core}{core}
\newcommand{\m}{\mathfrak m}
\newcommand{\frm}{\mathfrak{m}}

\newcommand{\PP}{\mathbb P}

\newcommand{\Q}{\mathbb Q}

\newcommand{\C}{\mathbb C}
\newcommand{\N}{\mathbb N}

\newcommand{\cal}{\mathcal}

\newcommand{\cC}{\mathcal C}

\newcommand{\cL}{\mathcal L}

\newcommand{\cO}{\mathcal O}

\newcommand{\cR}{\mathcal R}

\newcommand{\OO}{\mathcal{O}}

\newcommand{\rmH}{\mathrm{H}}

\renewcommand{\t}{\widetilde}

\renewcommand{\:}{\colon}

\newcommand{\pg}{p_g}
\begin{document}
\title{Good ideals and $p_g$-ideals in two-dimensional normal singularities}
\author{Tomohiro Okuma}
\address[Tomohiro Okuma]{Department of Mathematical Sciences,
Faculty of Science, Yamagata University, Yamagata, 990-8560, Japan.}
\email{okuma@sci.kj.yamagata-u.ac.jp}
\author{Kei-ichi Watanabe}
\address[Kei-ichi Watanabe]{Department of Mathematics, College of
Humanities and Sciences,
Nihon University, Setagaya-ku, Tokyo, 156-8550, Japan}
\email{watanabe@math.chs.nihon-u.ac.jp}
\author{Ken-ichi Yoshida}
\address[Ken-ichi Yoshida]{Department of Mathematics,
College of Humanities and Sciences,
Nihon University, Setagaya-ku, Tokyo, 156-8550, Japan}
\email{yoshida@math.chs.nihon-u.ac.jp}
\thanks{This work was partially supported by JSPS
Grant-in-Aid for Scientific Research (C)
Grant Numbers 23540068, 23540059, 25400050}
\begin{abstract}
In this paper, we introduce the notion of $p_g$-ideals and $p_g$-cycles,
which inherits nice properties of integrally closed ideals on rational singularities.
As an application, we prove an existence of
good ideals for two-dimensional Gorenstein normal local rings. Moreover,
we classify all Ulrich ideals for two-dimensional simple elliptic
singularities. 
\end{abstract}
\maketitle
\section{Introduction}
In a two-dimensional rational singularity, Lipman showed in \cite{Li}
that every integrally closed ideal is \lq\lq stable" in the sense that
$I^2 = IQ$ holds for every minimal reduction $Q$ of $I$. He also shows that
if $I,J$ are integrally closed ideals, then the product $IJ$ is also
integrally closed.
(Later, Cutkosky \cite{Cu} showed that in a 
two-dimensional normal local ring $A$,
if $I^2$ is integrally
closed for every integrally closed ideal $I$, then $A$ is a rational
singularity. )
These facts play very important role to study ideal theory on a two-dimensional rational singularity.
\par
On the other hand, as far as the authors know, almost nothing was done concerning
ideal theory of non-rational singularities.

\par
Let $(A,\frm)$ be a normal local ring of dimension $2$ and 
$ f: X \to \Spec A$ be a resolution of singularity of $A$.
Then $p_g(A) = \ell_A(H^1(X,\cO_X))$ is an important invariant of $A$
(here we denote by $\ell_A(M)$ the length of an $A$ module $M$
with finite length) and called the \lq\lq geometric genus" of $A$.
A rational singularity is characterized by $p_g(A)=0$ and
$A$ is a \lq\lq minimally elliptic singularity" if $A$ is Gorenstein and
$p_g(A)=1$.

\par
Now, take an integrally closed $\frm$-primary ideal $I$.
Then $I$ has a resolution $f : X\to \Spec A$ with $I\cO_X$ 
 invertible. In this case, $I\cO_X = \cO_X(-Z)$ for some \lq\lq anti-nef" cycle $Z$ and we denote
$I= I_Z$.
The important fact is that 
$\ell_A(H^1(X,\cO_X(-Z)))$ plays an important
role for the property of $I_Z$.
We can show that $\ell_A(H^1(X,\cO_X(-Z)))\le p_g(A)$ for every anti-nef
cycle $Z$ such that $\cO_X(-Z)$ has
no fixed component.
We call $Z$ a $p_g$-cycle and $I_Z$ a $p_g$-ideal
if we have $\ell_A(H^1(X,\cO_X(-Z)))= p_g(A)$.

\par
A surprising fact is that the class of $p_g$-ideal inherits nice
properties of integrally
closed ideals of rational singularities (in a rational singularity,
every integrally closed
ideal is a $p_g$-ideal by our definition).
Namely, if $I_Z$ is a $p_g$-ideal, then $I_Z$ is stable and if $Z,Z'$
are $p_g$-cycles,
then $I_ZI_{Z'}$ is integrally closed and
 also a $p_g$-ideal.
The idea of this paper is to develop ideal theory for normal two-dimensional local ring of
given $p_g$ and to investigate what difference causes the difference of $p_g$ to the ideal theory of the ring.

\par
We apply the notion of $p_g$-ideals to show existence of good ideals on
any two-dimensional
normal Gorenstein ring.
The notion of good ideals was defined by S. Goto and S. Iai in \cite{GIW}.

\begin{defn}[Goto--Iai--Watanabe \cite{GIW}]
Let $(A, \m)$ be a Cohen-Macaulay local ring, $I$ be an $\frm$-primary ideal of $A$
and $Q$ a minimal reduction of $I$.
We say $I$ is a good ideal if it satisfies the following conditions$:$
\begin{enumerate}
\item $I^2 = IQ$.
\item $Q:I=I$.
\end{enumerate}
\end{defn}

Let us explain the organization of the paper. In this paper, the ring is 
a two-dimensional normal local ring containing an algebraically closed field. 
In Section 2, we prepare the notions and terminologies which we need later
(e.g. minimally elliptic singularity, good ideals, Ulrich ideals and so on).  
Furthermore, we give fundamental tools in this paper: 
Propositions \ref{Ineq-h1}, \ref{vareps} and 
Vanishing theorem (Theorem \ref{t:Lv}), Kato's Riemann-Roch Theorem (Theorem \ref{t:kato}). 
\par 
In Section 3, we introduce the notion of $p_g$-cycles and $p_g$-ideals; 
see Theorem \ref{t:lepg}, Definition \ref{pgcycle}  and Lemma \ref{l:h^1}. 
Note that every anti-nef cycle in a rational singularity is a $p_g$-cycle in our sense.  
Moreover, $p_g$-cycles enjoy nice properties$:$ e.g. 
the sum of $p_g$-cycles is always a $p_g$-cycle (see Theorem \ref{t:sg}). 
\par 
In Section 4, we prove an existence of $p_g$-ideals for any two-dimensional normal local ring. 
As an application, we prove 
the existence of good ideals as the main theorem in this paper. 

\begin{thm}[See Theorem \ref{t:G}]
Let $(A,\m)$ be a two-dimensional normal local ring. Then$:$
\begin{enumerate}
\item There exists a resolution on which $p_g$-cycles exist.
\item If $A$ is non-regular Gorenstein, then it has a good ideal.
\end{enumerate}
\end{thm}
  
\par 
In Section 5,  we prove that an $\m$-primary ideal $I$ 
of a two-dimensional rational singularity $A$ is a good ideal if and only if 
it is an integrally closed ideal represented  on the minimal resolution, which 
is a generalization of \cite[Theorem 7.8]{GIW} 
for non-Gorenstein case. 

\par 
In Section 6, we evaluate the number of minimal generators of integrally 
closed ideals which is represented by some anti-nef cycle 
in terms of intersection numbers of related cycles. 

\par 
In Section 7, we investigate Ulrich ideals 
($3$-generated good ideals)  of 
minimally elliptic singularities. 
For instance, we prove that there exist no Ulrich ideals for any  minimally elliptic singularity of degree 
$e \ge 5$; see Theorem \ref{Ul-min-ell}. 
Moreover, we classify all Ulrich ideals for simple  elliptic singularities (Theorem \ref{Main-ell}).

\section{Preliminaries}

Throughout this paper, let $A$ be an excellent normal local ring of
dimension $2$
with the unique maximal ideal $\m$ such that $A$
contains an algebraically closed field $k \cong A/\m$ unless otherwise
specified.
Let $f \colon X \to \Spec A$ be a resolution of singularities
with exceptional divisor $E=f^{-1}(\m)$.

\subsection{Cycle} 
A divisor supported in $E$ is called a {\em cycle}.
Let $E=\bigcup_{i=1}^rE_i$ be the decomposition into irreducible
components of $E$.
A divisor $D$ is said to be {\em nef} if the intersection numbers $DE_i$
are nonnegative for all $E_i$;
$D$ is said to be {\em anti-nef} if $-D$ is nef.
If $DE_i=0$ for all $E_i$, then we say that $D$ is numerically trivial
and write $D\equiv 0$.
Since the intersection matrix $(E_iE_j)$ is negative definite, if a
cycle $Z\ne 0$ is anti-nef, then $Z\ge E$.
The resolution $f\: X\to \spec A$ is said to be {\em minimal} if $X$
contains no $(-1)$-curves $C$ 
(i.e., $C\cong \PP^1$, $C^2=-1$).

\subsection{Reduction, Multiplicity}
Let $I$ be an $\m$-primary ideal of $A$.
Then the Hilbert function $\ell_A(A/I^{n+1})$ is
a polynomial for sufficiently large $n$. That is,
there exists a polynomial $P_I(n)$ of the form
\[
e_0(I) \genfrac{(}{)}{0pt}{0}{n+2}{2}
-e_1(I)\genfrac{(}{)}{0pt}{0}{n+1}{1}
+e_2(I)
\]
such that $\ell_A(A/I^{n+1})=P_I(n)$ for $n \gg 0$.
Then $e_0(I)$, $e_1(I)$ and $e_2(I)$ are integers and
$e_0(I)$ is called the \textit{multiplicity} of $I$.
On the other hand,
we can take a parameter ideal $Q=(a,b)$ so that $I^{r+1}=QI^r$ for some
integer $r \ge 0$.
Such an ideal $Q$ is called a \textit{minimal reduction} of $I$.
Then we have $e_0(I)=e_0(Q)=\ell_A(A/Q)$.

\subsection{Integrally closed ideal}
\par
Let $\overline{I}$ denote the \textit{integral closure} of $I$, that is,
$\overline{I}$ is an ideal which consists of 
all solutions $z$
for some equation with coefficients $c_i \in I^i$:
$Z^{n}+c_1Z^{n-1}+\cdots +c_{n-1}Z+c_n=0$.
Then $I \subseteq \overline{I} \subseteq \sqrt{I}$.

\par
For any cycle $Z$ on $X$, we write
\begin{equation}
I_Z=H^0(\cO_X(-Z)).
\end{equation}
Since $A=H^0(\cO_X)$, $I_Z$ is an $\m$-primary integrally closed ideal if $Z>0$.
An $\m$-primary ideal $I$ is said to be {\em represented on} $X$ if the
ideal sheaf $I\cO_X$ is invertible and $I=H^0(I\cO_X)$.
If $I$ is represented on $X$, there exists an anti-nef cycle $Z$ such
that $I\cO_X=\cO_X(-Z)$; $I$ is also said to be {\em represented by} $Z$.
Note that $I$ is represented on 
some resolution if and only if it is
integrally closed (cf. \cite{Li}).
\par
Note that if $I=I_Z$ and $\mathcal{O}_X(-Z)$ is generated, then
$e_0(I)=-Z^2$ and $\overline{I^n}=I_{nZ}$.

\subsection{Geometric genus, Singularity}
When the cohomology group $H^i(\cal F)$ is an $A$-module,
we denote by $h^i(\cal F)$
the length $\ell_A(H^i(\cal F))$.
It is known that $h^1(\cO_X)$ is independent of the choice of the
resolution.
The invariant $p_g(A):=h^1(\cO_X)$ is called the {\em geometric genus}
of $A$.

\begin{defn}[\textbf{Rational singularity, Elliptic singularity}]
\label{Sing} 
A ring $A$ is said to be
\textit{a rational singularity}
(resp. \textit{a minimally elliptic singularity})
if $p_g(A)=0$ (resp. $A$ is Gorenstein and $p_g(A)=1$). 
\par 
Assume that $A$ is a minimally elliptic singularity, and 
let $Z_f$ be the fundamental cycle.  
Then $e=-Z_f^2$ is called the {\em degree} of $A$. 
It is known that $e_0(\m)=\max\{2,e\}$,  
$\ell_A(\m/\m^2)=\max\{e,3\}$
and that $A$ is a complete intersection if and only if $e \le 4$ (e.g. Laufer \cite{la.minell}). 

\par
A ring $A$ is said to be a
\textit{simple elliptic singularity of degree $e$}
if the exceptional set $E_0$ of the minimal resolution of 
$\Spec A$ is a nonsingular elliptic curve
with $E_0^2=-e$.
\end{defn}

\par
Every toric singularity and quotient singularity are
rational singularities.
For instance, $k[[x,y,z]]/(x^a+y^b+z^c)$ is rational if and only if
$1/a+1/b+1/c >1$.
Note that any simple elliptic singularity
is a minimally elliptic singularity. 

\begin{exam} 
\label{sim-ell-ex} 
Let $k$ be an algebraically closed field of
characteristic zero or $p \ge 5$.
\begin{enumerate}
\item A hypersurface $k[[x,y,z]]/(x^2+y^3+z^6)$
is a simple elliptic singularity
of degree $1$.
\item A hypersurface $k[[x,y,z]]/(x^2+y^4+z^4)$
is a simple elliptic singularity
of degree $2$.
\item A hypersurface $k[[x,y,z]]/(x^3+y^3+z^3)$
is a simple elliptic singularity
of degree $3$.
\item A complete intersection ring
$k[[x,y,z,w]]/(y^2-xz,w^2-yz-x^2)$ is a simple elliptic singularity of
degree $4$.
\end{enumerate}
\end{exam}

\subsection{Good ideal, Ulrich ideal}
In this subsection,
let $A$ be a Cohen-Macaulay local ring of any dimension $d$.

\begin{defn}[\textbf{Good ideal, Ulrich ideal}] 
\label{good-Ulrich} 
Let $I$ be an $\m$-primary ideal.
Then$:$
\begin{enumerate}
\item $I$ is called \textit{a stable ideal}
if $I^2=QI$ for some minimal reduction $Q$ of $I$.
\item $I$ is called \textit{a good ideal} if $I$ is stable and
$Q \colon I=I$ for some minimal reduction $Q$ of $I$.
\item $I$ is called \textit{an Ulrich ideal} if $I$ is stable and
$I/I^2$ is a free $A/I$-module.
\end{enumerate}
\end{defn}

Now assume that $I$ is stable.
Then $I \subset Q:I$ and $I^2\subset Q$ for any minimal
reduction $Q$ of $I$.
By the characterization of core of ideals,
a goodness of $I$ is equivalent to
the condition $\core(I)=I^2$
(see \cite[Example 3.1]{CPU}).
Recall that $\core(I)$ is the intersection of all minimal reductions of
$I$.
\par
If we assume that $A$ is a Gorenstein ring, then by duality theorem, we
have
\[
\ell_A(A/I) = \ell_A(Q:I/Q).
\]
Hence in this case, under the condition $I^2=QI$, $I$ is a good ideal if
and only if
$2 \cdot \ell_A(A/I) = e_0(I)$ (see \cite{GIW}).
Moreover, $I$ is an Ulrich ideal if and only if
$I$ is a good ideal with $\mu_A(I)=d+1$, where $\mu_A(I)$
denotes the cardinality of a minimal set of generators of $I$ (see
\cite{GOTWY1}).
So, in this case,
Ulrich ideals are typical examples of good ideals.

\par
We note a simple but useful lemma for good ideals.

\begin{lem}
\label{good_lem} 
Let $I'$ be an ideal containing $I$ and integral over
$I$ and assume that $I'^2 = QI'$ holds. Then if $I$ is a good ideal,
then $I=I'$.
In particular, if $A$ is a two-dimensional rational singularity and $I$
is a good
ideal of $A$, then $I$ is integrally closed.
\end{lem}

\begin{proof} Since $I' \subset Q:I'$,
we have $I' \subset Q:I'\subset Q:I =I$.
Hence $I=I'$.
\end{proof}

\subsection{Fundamental short exact sequences}
We say that $\cO_X(-Z)$ {\em has no fixed component}
if $H^0(\cO_X(-Z))\ne H^0(\cO_X(-Z-E_i))$ for every $E_i\subset E$,
i.e., the base locus of the linear system $H^0(\cO_X(-Z))$
does not contain any component of $E$.
\par
Suppose that $\cO_X(-Z)$ has no fixed component and $h\in I_Z$ a general
element.
Then we obtain the following exact sequence:
\begin{equation}\label{eq:OO(-Z)}
0 \to \cO_X \xrightarrow{\times h} \cO_Z(-Z) \to \cC \to 0,
\end{equation}
where $\cC$ is supported on the strict transform of the curve $\spec A/
(h)$.
Note that the base points of $H^0(\cO_X(-Z))$ is contained in $\supp \cC$.

\begin{prop} 
\label{Ineq-h1} 
Let $A$ be a two-dimensional normal local ring as above.
Let $Z$, $Z'$ be anti-nef cycles
on some resolution $X \to \Spec A$.
Suppose that $\cO_X(-Z)$ has no fixed components.
Then we have
\[
h^{1}(\mathcal{O}_X(-Z'))
\ge h^{1}(\mathcal{O}_X(-Z-Z')).
\]
\end{prop}

\begin{proof}
Let $h$ be a general element of $I_Z$.
Then the short exact sequence (\ref{eq:OO(-Z)})
implies that
\[
H^1(\mathcal{O}_X(-Z')) \to
H^{1}(\mathcal{O}_X(-Z-Z')) \to
H^{1}(\mathcal{C} \otimes \mathcal{O}_X(-Z'))=0
\]
since $\cC$ is a coherent sheaf on an affine space.
\end{proof}

\par
Let $Z_1$, $Z_2$ be anti-nef cycles on the resolution $X \to \Spec A$ 
so that $\cO_X(-Z_1)$ and $\cO_X(-Z_2)$ are generated. 
Take general elements $f_i \in I_{Z_i}$ for
each $i=1,2$, so that there exists the following exact sequence$:$
\begin{equation} \label{Secondeq}
0 \to \mathcal{O}_X
\stackrel{(f_1,f_2)}{\longrightarrow} \mathcal{O}_X(-Z_1) \oplus
\mathcal{O}_X(-Z_2)
\stackrel{\genfrac{(}{)}{0pt}{}{-f_2}{f_1}}{\longrightarrow}
\mathcal{O}_X(-Z_1-Z_2) \to 0.
\end{equation}
Taking a cohomology yields
\[
0 \to A \to I_{Z_1} \oplus I_{Z_2}
\stackrel{\genfrac{(}{)}{0pt}{}{-f_2}{f_1}}{\longrightarrow}
I_{Z_1+Z_2} \to H^1(\mathcal{O}_X).
\]
Hence we have the following.

\begin{prop} 
\label{vareps} 
Under the notation as above, if we put
\begin{equation}
\varepsilon(Z_1,Z_2):=\varepsilon(I_{Z_1},I_{Z_2}):=
\ell_A(I_{Z_1+Z_2}/f_1I_{Z_2}+f_2I_{Z_1}),
\end{equation}
then we have
\begin{enumerate}
\item
$0 \le \varepsilon(Z_1,Z_2) \le p_g(A)$.
\item
$\varepsilon(Z_1,Z_2)
= p_g(A)-h^1(\mathcal{O}_X(-Z_1))-h^1(\mathcal{O}_X(-Z_2))
+h^1(\mathcal{O}_X(-Z_1-Z_2))$.
\end{enumerate}
In particular, $\varepsilon(Z_1,Z_2)$ is independent on
the choice of general elements $f_1 \in I_1$,$f_2\in I_2$.

\par
If $Z_1=Z_2=Z$, then $Q=( f_1, f_2)$ is a minimal reduction of $I=I_1=I_2$ and 
$\varepsilon(Z,Z)=\varepsilon(I,I)=
\ell_A(\overline{I^2}/QI)$.
\end{prop}

\subsection{Canonical divisor, Vanishing theorem}
Let $K_X$ denote the canonical divisor on $X$.
Since the intersection matrix $(E_iE_j)$ is
negative-definite, there exists a $\Q$-divisor $Z_{K_X}$ supported in
$E$ such that $K_X+Z_{K_X}\equiv 0$.
It is known that: $Z_{K_X}\ge 0$ if $X$ is the minimal resolution;
$Z_{K_X}=0$ if and only if $A$ is rational
Gorenstein and $X$ is the minimal resolution; $K_X$ is linearly
equivalent to $-Z_{K_X}$
if and only if $A$ is Gorenstein.
\par
The following theorem is a generalization of 
Grauert--Riemenschneider
vanishing theorem in two dimensional case.

\begin{thm}[Laufer {\cite[Theorem 3.2]{la.rat}}, 
cf.{\cite[Ch. 4, Exe.15]{chap}}] 
\label{t:Lv} 
For any nef divisor $D$ on $X$,
we have $H^1(X,\cO_X(K_X+D))=0$.
\end{thm}

\subsection{Riemann-Roch formula}
Let us recall Kato's Riemann-Roch formula which is very useful in order
to calculate colength.
For any invertible sheaf $\cL$ on $X$,
we define $\chi (\cL)$ by
\[
\chi(\cL)
=\ell_A\left(H^0(X\setminus E, \cL)/H^0(X,\cL)\right)
+h^1(\cL).
\]
Note that $\chi(\cO_X)=p_g(A)$ since $A$ is normal.

\begin{thm}[\textbf{Kato's Riemann-Roch formula}{\cite{kato}}]
\label{t:kato} 
For a cycle $Z>0$, we have
\[
\chi(\cL(-Z))-\chi(\cL)=-\frac{Z^2+K_XZ}{2}+\cL Z.
\]
In particular,
\[
\ell_A(A/I_Z) + h^1(\cO_X(-Z))=-\dfrac{Z^2+K_XZ}{2}+p_g(A).
\]
\end{thm}
\section{$p_g$-cycles and $p_g$-ideals}

The main aim of this section is to introduce the notion
of $p_g$-cycles and $p_g$-ideals.
We first show the following theorem, which is the key result in this paper.

\begin{thm}
\label{t:lepg} 
Let $Z>0$ be a cycle.
Suppose that $\cO_X(-Z)$ has no fixed component.
Then we have the following.
\begin{enumerate}
\item $h^1(\cO_X(-Z))\le p_g(A)$.
\item If $h^1(\cO_X(-Z))=p_g(A)$,
then $\cO_X(-Z)$ is generated $($by global sections$)$.
\end{enumerate}
\end{thm}

\begin{proof}
We use the exact sequence \eqref{eq:OO(-Z)}.
\par \noindent
(1) It follows from Proposition \ref{Ineq-h1} because
$h^1(\cO_X)=p_g(A)$.
\par \noindent
(2)
If $h^1(\cO_X(-Z))=p_g(A)$, then the restriction $H^0(\cO_Z(-Z)) \to
H^0(\cC)$ is surjective. 
This implies that $H^0(\cO_X(-Z))$ has no base
points.
\end{proof}

\begin{defn}[\textbf{$p_g$-cycle, $p_g$-ideal}] 
\label{pgcycle} 
A cycle $Z>0$ is called a {\em $p_g$-cycle} if $\cO_X(-Z)$ is generated
and $h^1(\cO_X(-Z))=p_g(A)$.
An $\m$-primary ideal $I$ is called a {\em $p_g$-ideal} if $I$ is represented by a $p_g$-cycle on some resolution.
The definition of $p_g$-ideal is independent of the representation of
the ideal by \lemref{l:h^1}.
\end{defn}

\begin{ex} 
\label{Lipman} 
If $A$ is rational, then every anti-nef cycle
is a $p_g$-cycle.
In fact, Lipman \cite{Li} proved that if $A$ is rational and $Z>0$ is an
anti-nef cycle on $X$,
then $\cO_X(-Z)$ is generated and $H^1(\cO_X(-Z))=0$.
\end{ex}

\par
A birational morphism $\phi\: Y\to \spec A$ is called a {\em partial
resolution}
if $Y$ is normal and $\phi$ induces an isomorphism
$Y\setminus\phi^{-1}(\m)\cong \spec A\setminus\{\m\}$.

\begin{lem}
\label{l:h^1} 
Let $I$ be an $\m$-primary ideal, and
let $f_1\:X_1\to \spec(A)$ and $f_2\:X_2\to \spec(A)$
be partial resolutions with only rational singularities.
Assume that $I$ is represented by a cycle $Z_i$ on $X_i$ for $i=1,2$.
Then $h^1(\cO_{X_1}(-Z_1))
=h^1(\cO_{X_2}(-Z_2))$.
\end{lem}

\begin{proof}
Take a resolution $f_3\:X_3\to \spec A$ which factors through
$f_1$ and $f_2$ as follows:
\[
\begin{array}{ccc}
X_3 & \stackrel{\phi_2}{\longrightarrow} & X_2 ~ \\
{}_{\phi_1} \!\downarrow ~ & & \downarrow {}_{f_2} \\
X_1 & \stackrel{f_1}{\longrightarrow} & \Spec A
\end{array}
\]
Then $\phi_i$ are resolution of singularities of $X_i$,
and $\phi_1^*Z_1=\phi_2^*Z_2$ because they are determined by the
invertible sheaf $I\cO_{X_3}$.
Let $Z_3=\phi_1^*Z_1$.
From the Leray spectral sequence, we obtain the following exact sequence:
\[
0\to H^1(\phi_{i*}\cO_{X_3}(-Z_3))
\to H^1(\cO_{X_3}(-Z_3))
\to H^0(R^1\phi_{i*}\cO_{X_3}(-Z_3)).
\]
By projection formula,
$R^j\phi_{i*}\cO_{X_3}(-Z_3)=\cO_{X_i}(-Z_i)
\otimes R^j\phi_{i*}\cO_{X_3}$.
Since $X_i$ has only rational singularities,
we have $R^1\phi_{i*}\cO_{X_3}(-Z_3)=0$ and
$\phi_{i*}\cO_{X_3}(-Z_3)=\cO_{X_i}(-Z_i)$.
Thus we obtain
that $h^1(\cO_{X_i}(-Z_i))=h^1(\cO_{X_3}(-Z_3))$ for $i=1,2$.
\end{proof}

\par
Any $p_g$-ideal is an integrally closed $\m$-primary ideal
by definition. Indeed, all powers of $p_g$-ideals is
$p_g$-ideals and thus integrally closed.

\begin{thm} 
\label{t:sg} 
Assume that $Z$ is a $p_g$-cycle on the resolution $X$. 
Then for any cycle $Z'$ on $X$ such that $\cO_X(-Z')$ is generated, 
$\varepsilon( Z, Z') =0$. In particular, 
$Z'$ is a $p_g$-cycle if and only if
so is $Z+Z'$. 
\par
When this is the case, 
if $f \in I_Z$, $f' \in I_{Z'}$ are general
elements, then
\[
I_{Z+Z'}=fI_{Z'}+f'I_Z.
\]
\end{thm}

\begin{proof}
Consider
\[
\varepsilon(Z,Z')=p_g(A)-h^1(\cO_X(-Z))-
h^1(\cO_X(-Z'))+h^1(\cO_X(-Z-Z')).
\]
Assume $Z$ is a $p_g$-cycle.
Then $\varepsilon(Z,Z')=
-h^1(\cO_X(-Z'))+h^1(\cO_X(-Z-Z')) \le 0$ by
Lemma \ref{Ineq-h1}.
On the other hand, as $\varepsilon(Z,Z') \ge 0$, 
we obtain that $\varepsilon(Z,Z')=0$, that is,
$h^{1}(\cO_X(-Z'))=h^1(\cO_X(-Z-Z'))$.
Hence $Z'$ is a $p_g$-cycle if and only if $Z+Z'$ is
a $p_g$-cycle.
\end{proof}

\begin{cor}
\label{t:stable} 
Let $Z$ be a $p_g$-cycle on $X$ and $Q$ a minimal reduction of $I:=I_Z$.
Then $I^n$ is integrally closed for all $n \ge 1$, $I^2=IQ$, and
$I\subset Q \colon I$.
\end{cor}

\begin{proof}
We can apply the previous theorem as $Z_1=Z_2=Z$.
\end{proof}

\begin{rem} In our upcoming paper, we will prove that for an $\frm$ primary ideal $I$ in a 
two-dimensional normal local ring $A$, the Rees algebra $\cR(I) = \oplus_{n\ge 0} I^n t^n$ is 
normal and Cohen-Macaulay if and only if $I$ is a 
$p_g$-ideal.  
\end{rem} 

In the rest of this section, 
we give a characterization of $p_g$-cycles.
\par
For any cycle $D>0$ on $X$, the restriction
$\cO_X\to \cO_{D}$ implies the surjection 
$H^1(\cO_X)\to H^1(\cO_{D})$.
Thus $h^1(\cO_{D})\le p_g(A)$.

\begin{thm}[{Reid \cite[\S 4.8]{chap}}]
\label{t:cohom} 
Assume that $p_g(A)>0$.
There exists a smallest cycle $C_X>0$ on $X$ such that
$h^1(\cO_{C_X})=p_g(A)$.
If $A$ is Gorenstein and $f$ is minimal, then $C_X=Z_{K_X}$.
\end{thm}

The cycle $C_X$ is called the {\em cohomological cycle}
on $X$.

\begin{defn} 
\label{perp} 
For any cycle $D$ on $X$, let $D^{\bot}=\sum_{DE_i=0}E_i$.
\end{defn}

\begin{prop}
\label{p:CX} 
Assume that $p_g(A)>0$.
Let $Z>0$ be a cycle such that $\cO_X(-Z)$ has no fixed component.
Then $Z$ is a $p_g$-cycle if and only if
$\cO_{C_X}(-Z)\cong \cO_{C_X}$.
\end{prop}

\begin{proof}
If $\cO_{C_X}(-Z)\cong \cO_{C_X}$,
then $h^1(\cO_{X}(-Z)) \ge h^1(\cO_{C_X})=p_g(A)$.
By \thmref{t:lepg}, $h^1(\cO_{X}(-Z))=p_g(A)$
and $\cO_X(-Z)$ is generated.
\par
Conversely, assume that $h^1(\cO_{X}(-Z))=p_g(A)$.
Then $h^1(\cO_{X}(-mZ))=p_g(A)$ and $\cO_{X}(-mZ)$ is generated for
every $m \in \N$ by \thmref{t:sg}. .
If $Z^{\bot}=0$, then it follows from \thmref{t:Lv} that
$H^1(\cO_X(-mZ))=0$ for sufficiently large
$m\in \N$; it contradicts that $p_g(A)>0$.
Let $D>0$ be a cycle supported on $Z^{\bot}$ and
$DE_i<0$ for all $E_i\le Z^{\bot}$.
There exist $m\in \N$ such that $(mZ+D)E_i<0$ for every $E_i\subset E$.
By \thmref{t:Lv} again, $H^1(\cO_X(-nmZ-nD))=0$ for some $n\in \N$.
Then $p_g(A)=h^1(\cO_X(-nmZ))=h^1(\cO_{nD}(-nmZ))$.
Since $\cO_X(-Z)$ is generated and $Z\equiv 0$ on $D$,
we have $\cO_{nD}(-Z)\cong \cO_{nD}$.
It follows that
$h^1( \cO_{nD})= h^1(\cO_{nD}(-nmZ))=p_g(A)$.
By the definition of $C_X$, we have $nD\ge C_X$.
Hence $\cO_{C_X}(-Z)\cong \cO_{C_X}$.
\end{proof}

It follows from \thmref{t:cohom} and \proref{p:CX} that
if $A$ is Gorenstein, $p_g(A)>0$, and $f\: X\to \spec A$ is minimal,
then there exist no $p_g$-cycles on $X$.
Therefore, in general, $p_g$-ideals are represented on non-minimal
resolutions.
In the next proposition, we discuss the minimality of representation.

\begin{prop} 
\label{minrep} 
Let $I$ be a $p_g$-ideal represented by a cycle
$Z$ on $X$.
Then there exist the minimum $X_1\to \spec A$ of the resolutions on
which $I$ is represented and a natural morphism $X\to X_1$.
We call $X_1$ the
{\em minimal resolution with respect to $I$}.
The resolution $X$ is the minimum with respect to $I$
if and only if $ZC<0$ for every $(-1)$-curve $C$ on $X$.
\end{prop}

\begin{proof}
Let $X_0\to \spec A$ be a partial resolution obtained by normalizing the
blowing-up by the ideal $I$.
Since $I\cO_X=\cO_X(-Z)$, $X_0$ is also obtained by contracting all
curves $E_i\subset E$ with $ZE_i=0$;
let $\psi\:X\to X_0$ denote the contraction.
If $I$ is represented on a resolution $X'\to \spec A$,
then $I\cO_{X'}$ is invertible and thus there exists a unique morphism
$X'\to X_0$ by universal property of blowing-ups.
Hence the minimal resolution with respect to $I$ is obtained as the
minimal resolution of singularities of $X_0$.
\par
Let $C$ be a $(-1)$-curve on $X$ with $ZC=0$ and
let $X\to X'$ be the contraction of $C$.
Then $X'\to \spec A$ is a resolution, and
$I$ is represented on $X'$ since we have a morphism $X' \to X_0$.
Conversely assume that the natural morphism
$\psi_1\: X\to X_1$ to the minimal resolution
$X_1$ with respect to $I$ is not trivial.
Then the exceptional set of $\psi_1$ contains a $(-1)$-curve $C$,
and the invertible sheaf $\cO_X(-Z)$ is trivial
on $C$, since $\cO_X(-Z)=I\cO_X=\psi_1^*I\cO_{X_1}$.
\end{proof}

\section{Existence of good ideals in two-dimensional normal Gorenstein singularities}

The aim of this section is to prove the following,
which is the main theorem in this paper.

\begin{thm}
\label{t:G} 
Let $(A,\m)$ be a two-dimensional normal local ring. Then$:$
\begin{enumerate}
\item There exists a resolution on which $p_g$-cycles exist.
\item If $A$ is non-regular Gorenstein, then it has a good ideal.
\end{enumerate}
\end{thm}

\par
This theorem follows from Propositions
\ref{p:Ggood}, \ref{p:exs} below.
We use the notation of the preceding sections.

\begin{prop}
\label{p:Ggood} 
Assume that $A$ is Gorenstein.
Let $Z>0$ be a $p_g$-cycle on $X$.
Then $I:=I_Z$ is a good ideal if and only if $K_XZ=0$.
\end{prop}

\begin{proof}
Applying \thmref{t:kato}, we have
$2 \cdot \ell_A(A/I)=-Z^2-K_XZ$.
As noted in subsection 2.5,
$I$ is good if and only if
$I^2=IQ$ and $2 \cdot \ell_A(A/I)=e_0(I)$
for some minimal reduction $Q$ of $I$.
However the condition $I^2=IQ$ is always satisfied for $p_g$-ideals
by Corollary \ref{t:stable}.
Since $\cO_X(-Z)$ is generated, $e_0(I)=-Z^2$.
This completes the proof.
\end{proof}

\begin{lem}
\label{l:coh} 
Suppose that $C>0$ is a cycle on $X$ such that $h^1(\cO_C)=p_g(A)$.
Let $b\: Y\to X$ be the blowing-up of a finite subset $B\subset \supp (C)$.
Let $F=b^{-1}(B)$ and $\t C=b^*C-F$.
Then $h^1(\cO_{\t C})=p_g(A)$.
\end{lem}
\begin{proof}
It suffices to show that $h^1(\cO_{\t C})\ge p_g(A)$.
We have that $R^1b_*\cO_Y(-b^*C)=0$
and $H^i(\cO_F(-b^*C+F))=0$ for $i=0,1$,
since $\cO_F(-b^*C)\cong \cO_F$ and $F$ is
the sum of $(-1)$-curves.
From the exact sequence
\[
0 \to \cO_F(-b^*C+F)\to \cO_{b^*C}\to \cO_{\t C}\to 0,
\]
we obtain that $h^1(\cO_{b^*C})=h^1(\cO_{\t C})$.
On the other hand, it follows from the exact sequence
\[
0\to b_*\cO_Y(-b^*C)\to b_*\cO_Y\to b_*\cO_{b^*C}\to 0
\]
that $b_*\cO_{b^*C}=\cO_C$, since $b_*\cO_Y(-b^*C)
=\cO_X(-C)$ and $b_*\cO_Y=\cO_X$.
Therefore,
\[
p_g(A)=h^1(\cO_C)=h^1(b_*\cO_{b^*C})
\le h^1(\cO_{b^*C})=h^1(\cO_{\t C}).
\qedhere
\]
\end{proof}

\begin{rem} 
\label{Lem4.3rem} 
In the situation above, if $C'\subset Y$
denote the strict transform of $C$,
the equality $h^1(\cO_{C'})=h^1(\cO_C)$
does not hold in general.
Let $A=\C[[x,y,z]]/(x^2+y^3+z^7)$ and $X$
the minimal good resolution, i.e.,
$E$ is simple normal crossing and any $(-1)$-curve
intersects at least other three exceptional curves.
Then $E=E_0+\cdots+E_3$ is star-shaped,
where $E_0$ denotes the central curve, and $Z_{K_X}=E_0+E$.
If $b\: Y\to X$ is the blowing-up of a point of
$E_0\setminus (E_1\cup E_2\cup E_3)$,
then the strict transform of $E$ contracts to
a rational singularity.
If $C=Z_{K_X}$, then $h^1(\cO_C)=1$ but $h^1(\cO_{C'})=0$.
\end{rem}

\begin{prop}
\label{p:exs} 
There exist a resolution $g\: Y\to \spec A$ and a $p_g$-cycle $Z$ on $Y$.
Furthermore, such a resolution can be obtained from $X$ by
blowing-ups of smooth points of the exceptional set.
If $Z_{K_X}$ is a cycle
$($i.e., all the coefficients are integers$)$
and $Z_{K_X}> 0$, then $Z$ can be taken as a $p_g$-cycle
satisfying $ZK_Y=0$.
\end{prop}
\begin{proof}
Since the intersection matrix is negative definite,
there exists an anti-nef cycle $W>0$ such that for any $E_i$,
\[
-WE_j\ge \max\{K_XE_j, (K_X+E_i)E_j, 2 \cdot g(E_j)\}
\]
for every $E_j\subset E$.
Consider the following exact sequence:
\[
0\to \cO_X(-W-E_i) \to \cO_X(-W) \to \cO_{E_i}(-W) \to 0.
\]
Since
$-W-(K_X+E_i)$ is nef, it follows from \thmref{t:Lv} that
$H^1(\cO_X(-W-E_i))=0$.
Therefore the map
\[
H^0(\cO_X(-W)) \to H^0(\cO_{E_i}(-W) )
\]
is surjective.
If $\cO_X(-W)$ has a base point $p\in E_i$, then $p$ should also be a
base point of
$H^0(\cO_{E_i}(-W) )$.
On the other hand, since $\deg \cO_{E_i}(-W) \ge 2g(E_i)$,
the linear system $H^0(\cO_{E_i}(-W) )$ has no base points.
Thus we obtain that $\cO_X(-W)$ is generated.
Since 
$-W-K_X$ is nef, we have $H^1(\cO_X(-W))=0$ from \thmref{t:Lv}. Hence
the exact sequence
\[
0\to \cO_X(-W) \to \cO_X \to \cO_{W} \to 0
\]
implies $h^1(\cO_W)=p_g(A)$.
Since $\cO_X(-W)$ is generated, there exists a function
$h\in H^0(\cO_X(-W))$ such that $\di (h)=W+H_0$,
where $H_0$ is a reduced divisor including no component of $E$, and that
$E+H_0$ is normal crossing at $E\cap H_0$.
Let $C_0>0$ be a cycle on $X$ such that $h^1(\cO_{C_0})=p_g(A)$;
at least the cycle $W$ satisfies this property
(but $\cO_X(-C_0)$ need not be generated).
Let $Y_0=X$ and $B_0=\supp(C_0)\cap H_0$.
For $i\ge 0$, if $B_i\ne \emptyset$,
take the blowing-up $b_{i+1}\:Y_{i+1}\to Y_i$ of $B_i$, and
let $H_{i+1}$ be the strict transform of $H_i$ by $b_{i+1}$,
$C_{i+1}=b_{i+1}^*C_i-F_{i+1}$, where $F_{i+1}=b_{i+1}^{-1}(B_i)$,
and $B_{i+1}=\supp (C_{i+1})\cap H_{i+1}$.
Note that there exists $n$ such that $B_n=\emptyset$;
in fact, such $n$ is not more than the maximal coefficient of $C_0$.
If $B_n=\emptyset$, let $Y=Y_n$ and $Z=\di_Y(h)-H_n$.
Then $\cO_{C_n}(-Z)\cong \cO_{C_n}(-\di_Y(h))\cong \cO_{C_n}$.
By \lemref{l:coh}, $h^1(\cO_{C_n}(-Z))=p_g(A)$.
From the surjection
$H^1(\cO_{Y}(-Z)) \to H^1(\cO_{C_n}(-Z))$ and
\thmref{t:lepg}, we obtain that
$h^1(\cO_Y(-Z))=p_g(A)$ and $\cO_Y(-Z)$ is generated.
\par
If $Z_{K_X}$ is a cycle and $Z_{K_X}>0$, then we can take $C_0=Z_{K_X}$.
In this case we obtain that $C_n=Z_{K_Y}$.
Since $\cO_{C_n}(-Z)\cong \cO_{C_n}$, we obtain that $K_YZ=0$.
\end{proof}

Let us explain the procedure of Proposition \ref{p:exs} by an example.

\begin{ex} 
\label{ex:highpg} 
Let us consider a cone singularity and check the key roles in the proof of \proref{p:exs}.
Let $C$ be a nonsingular curve of genus $g\ge 2$ and put
\[
A=\bigoplus_{n\ge 0}H^0(\cO_C(nK_C)).
\]
By Pinkham's formula \cite{p.qh}, we have
$p_g(A)=\sum_{n\ge 0}h^1(\cO_C(nK_C))=g+1$.
Let $f\: X\to \spec A$ be the minimal resolution.
Then $E\cong C$, $\cO_E(-E)\cong \cO_E(K_E)$, and $K_X=-2E$.
It follows that $H^1(\cO_X(-2E))=0$ by \thmref{t:Lv}.
From the exact sequence
\[
0\to \cO_X(-2E)\to \cO_X(-E)\to \cO_E(-E)\to 0,
\]
we see that $\cO_X(-E)$ is generated.
Thus we can take $W=kE$ ($k\ge 1$) 
 and $C_0=Z_{K_X}=2E$. 
Then $B_0$ and $B_1$ consists of $-WE=k(2g-2)$ 
points and $B_2=\emptyset$. Hence we have $Y=Y_2$.
In this case, $-Z^2 = H_n Z = k(k+2)(2g-2)$.  
\par
The case $g=2$ can be realized by $A=k[x,y,z]/(x^2+y^6+z^6)$.
This is graded by $\deg x=3$, $\deg y=\deg z=1$.
If $W=E$ ($k=1$) and $h=y\in H^0(\cO_X(-W))$, 
then $I_Z=(x,y,z^3)$ and $I_Z$ is a good ideal with multiplicity $6$; this is also an Ulrich ideal (see Section 7). 
Let $W=2E$ ($k=2$) and $h=yz\in H^0(\cO_X(-W))$. 
Then $I:=I_Z$ is a homogeneous
ideal and
$I=(yz, y^4, z^4, xy, xz)$ is a good ideal with multiplicity $16$.
\end{ex}

\section{Good ideals for non-Gorenstein rational singularities}

In this section, we characterize good ideals for rational singularities,
which
gives a generalization of \cite[Section 7]{GIW}.

\begin{thm} 
\label{rational} 
Assume that $(A,\m)$ is a two-dimensional rational singularity.
Let $I$ be an $\m$-primary ideal of $A$.
Then the following conditions are equivalent:
\begin{enumerate}
\item $I$ is a good ideal, that is, $I^2=QI$ and $I=Q\colon I$ for some
minimal reduction $Q$ of $I$.
\item $I$ is an integrally closed ideal that is represented on the
minimal resolution.
\end{enumerate}
\end{thm}

\par
Now let $I$ be a good ideal 
in a rational singularity $(A,\m)$.
Then \lemref{good_lem} implies that $I$ is 
 integrally closed and thus
$I$ is represented by some anti-nef cycle $Z$
on some resolution of singularities $X \to \Spec A$.
Then $Z$ is a $p_g$-cycle on $X$; see Example \ref{Lipman}.
Before proving that $X$ is minimal, 
we study properties of $p_g$-cycle for
any normal local ring.

\par
We first show that any integrally closed ideal
that is represented on non-minimal resolution is \textit{not} good.
We need the following lemma.

\begin{lem}
\label{l:(-1)}
Let $E_1\le E$ be a $(-1)$-curve.
Then $h^1(\cO_X(E_1))=p_g(A)$.
\end{lem}

\begin{proof}
Consider the exact sequence
\[
0\to \cO_X\to \cO_X(E_1)\to \cO_{E_1}(E_1)\to 0.
\]
Since $H^i(\cO_{E_1}(E_1))=H^i(\cO_{\mathbb{P}^1}(-1))=0$ for $i=0,1$,
we have $h^1(\cO_X)=h^1(\cO_X(E_1))$.
\end{proof}

\par
The implication $(1)\Longrightarrow (2)$ in
Theorem \ref{rational} follows from the following
proposition and Lemma \ref{good_lem}.

\begin{prop} 
\label{blowdown}
Assume that $Z$ is a $p_g$-cycle on $X$
and there exists a $(-1)$-curve $E_1$ such that
$ZE_1=-a<0$.
Let $\phi\: X\to X'$ be the blowing-down of $E_1$
and $Z'=Z-aE_1$. Consider the following conditions$:$
\begin{enumerate}
\item $Z'$ is a $p_g$-cycle on $X$;
\item $\phi_*Z$ is a $p_g$-cycle on $X'$;
\item $I:=I_Z \subsetneqq I_{Z'}$.
\end{enumerate}
We have the implication:
$(1) \Leftrightarrow (2) \Rightarrow (3)$; if $a=1$,
$(3)$ implies $(1)$.
If the condition $(1)$ is satisfied, then
$I_{Z'}\subset Q: I$;
in particular, $I$ is not good.
If $(A,\m)$ is rational, then all conditions are satisfied.
\end{prop}

\begin{proof}
Note that $Z'=\phi^*\phi_*Z$ and
$\phi_*\cO_X(-Z')=\cO_{X'}(-\phi_*Z)$.
Thus $\cO_X(-Z')$ is generated if and only if so is
$\cO_{X'}(-\phi_*Z)$.
From the spectral sequence, we have that
\[
h^1(\cO_{X'}(-\phi_*Z))=h^1(\cO_X(-Z')).
\]
Therefore
the conditions (1) and (2) are equivalent.
Consider the exact sequence
\[
0\to \cO_X(-Z)\to \cO_X(-Z')\to \cO_{aE_1}(-Z')=\cO_{aE_1}\to 0.
\]
If (1) is satisfied, then the natural homomorphism
$\alpha\: H^0(\cO_X(-Z'))\to H^0(\cO_{aE_1})$ is surjective,
and (3) holds.
If $a=1$, the following three conditions are equivalent:
\begin{itemize}
\item $I\subsetneqq I_{Z'}$;
\item $\alpha$ is surjective;
\item $h^1(\cO_X(-Z))=h^1(\cO_X(-Z'))$.
\end{itemize}
Since the non-triviality of $\alpha$ implies that
$\cO_X(-Z')$ has no fixed components in $E$, $(3)$ implies $(1)$.
\par
Assume that $Z'$ is a $p_g$-cycle on $X$.
Then $Z+Z'$ is also a $p_g$-cycle on $X$
by Theorem \ref{t:sg}.
The following sequence is obtained
from (\ref{Secondeq}).
\[
0\to \cO_X(E_1)\to \cO_X(-Z')^2\to \cO_X(-Z-Z')\to 0.
\]
By Lemma \ref{l:(-1)}, $H^0(\cO_X(-Z')^2) 
\to H^0(\cO_X(-Z-Z'))$ is
surjective.
Thus $\overline{II_{Z'}}=QI_{Z'}$. This implies
$I_{Z'}\subset Q:I$.
\par
If $(A,\m)$ is rational, then $Z'$ is a $p_g$-cycle
from Example \ref{Lipman}.
\end{proof}

\begin{ex} 
\label{237}
Assume that $A=k[[x,y,z]]/(x^2+y^3+z^7)$ and $X'$ is
the minimal good resolution.
Then the exceptional set $E'=E'_0+\cdots+E'_3$ is star-shaped and all $E_i$ are rational curves.
Suppose that $E_0'$ is the central curve and $(-E_0'^2,\dots, -E_3'^2)=(1,2,3,7)$.
Then the canonical cycle on $X'$ is $Z_{K_{X'}}=E'_0+E'$.
Let $Z_0$ be the fundamental cycle on $X'$.
Then $\cO_{X'}(-Z_0)$ has no fixed components in $E'$,
and has a base point $p\in E_3'\setminus E_0'$.
Let $\phi\: X\to X'$ be the blowing-up of the point $p$ and $E_4$
the exceptional set, and let
$Z'=\phi^*Z_0$ and $Z=Z'+E_4$.
Then $Z_{K_X}=\phi_*^{-1}Z_{K_{X'}}$ and $\cO_{Z_{K_X}}(-Z)\cong
\cO_{Z_{K_X}}$.
Since $\cO_{X}(-Z)$ is generated and
$h^1(\cO_{X}(-Z)\ge h^1(\cO_{Z_X})=p_g$, $Z$ is a $p_g$-cycle
and $I_Z$ is good by Proposition \ref{p:Ggood}.
However, $I_Z=I_{Z'}$.
\end{ex}

\par
In what follows, we prove $(2) \Longrightarrow (1)$ in the theorem.

\begin{lem}
\label{l:D12}
Assume that $(A,\m)$ is rational and that
$f\:X \to \spec(A)$ is minimal.
Let $D_1$ and $D_2$ be effective cycles.
Suppose that they have no common irreducible components
and $D_1\ne 0$.
Then $H^1(\cO_X(D_1-D_2))\ne 0$.
\end{lem}

\begin{proof}
Consider the exact sequence
\[
0\to \cO_X(D_1-D_2) \to \cO_X(D_1) \to \cO_{D_2}(D_1) \to 0.
\]
Since $\cO_{D_2}(D_1)$ is nef on its support, $H^1(\cO_{D_2}(D_1))=0$.
Therefore, it suffices to show that $H^1(\cO_X(D_1))\ne 0$.
Since $H^0(\cO_{D_1}(D_1))=0$ (Wahl \cite[(2.2)]{wahl.va}), we have that
\[
h^1(\cO_{D_1}(D_1))=-\chi(\cO_{D_1}(D_1))=
(K_XD_1-D_1^2)/2.
\]
Since $X$ is minimal and $D_1\ne 0$, we have
$h^1(\cO_{D_1}(D_1))\ne 0$.
Therefore
$h^1(\cO_X(D_1))\ge h^1(\cO_{D_1}(D_1))\ge 1$.
\end{proof}

Assume that $A$ is a rational singularity.
Let $Z>0$ be an anti-nef cycle and $I=H^0(\cO_X(-Z))$.
Let $Q$ be a minimal reduction of $I$.
From the Koszul complex associated with generators of $Q$ (cf.(2.3)),
we obtain the exact sequence
\begin{equation}
\label{eq:Kos}
0 \to \cO_X(Z) \to \cO_X^2\to \cO_X(-Z) \to 0.
\end{equation}
This implies the exact sequence
\begin{equation}
\label{eq:QI}
0\to Q \to I \to H^1(\cO_X(Z))\to 0.
\end{equation}

The next lemma holds without rationality of the singularity.

\begin{lem}\label{l:core}
Assume that $I^2=QI$. Then 
$$
\core(I)=Q^2:I=(Q:I)I=(Q:I)Q.
$$
\end{lem}
\begin{proof}
By Goto--Shimoda \cite{Goto-Shimoda}, the Rees algebra $\cR(I) = \oplus_{n\ge 0} I^n t^n$ is Cohen-Macaulay.
Therefore it follows from Corollary 5.1.1 and Remark 5.1.2 of 
Hyry--Smith \cite{Hyry-Smith} that $\core(I)=Q^2:I$.
Let $x=\sum_{i=1}^ma_ih_i\in (Q:I)I$, where $a_i\in Q:I$ and $h_i\in I$.
Then $xI\subset \sum a_iI^2=\sum a_iIQ\subset Q^2$. 
Thus $(Q:I)I\subset Q^2:I$.
Conversely assume that $x\in Q^2:I$.
Suppose that $Q$ is generated by $f$ and $g$.
Since $\core(I)\subset Q$, there exist $a,b\in A$ such that $x=af+bg$.
For any $h\in I$, there exist $c_1,c_2,c_3\in A$ such that
$(af+bg)h=c_1 f^2+c_2 fg+c_3 g^2$. Since $f,g$ form a regular sequence, 
we have $ah-c_1f\in (g)$ and $bh-c_3g\in (f)$.
Thus $ah,bh\in (f,g)= Q$. This shows that $a,b\in Q:I$.
Hence $x=af+bg\in (Q:I)Q$.
\end{proof}

\begin{prop} 
\label{good_on_min}
Assume that $(A,\m)$ is rational and that
$f\:X \to \spec(A)$ is minimal.
Let $Z > 0$ be an anti-nef cycle on $X$ and $I=I_Z$.
Then $I^2=QI$ and $Q: I=I$.
\end{prop}

\begin{proof}
By Proposition 
3.6, it suffices to show that $Q: I\subset I$.
Let $h\in Q: I$ and $Z_h=\di(f^*h)-f_*^{-1}(\di(h))$ (the exceptional
part of the divisor of $h$ on $X$).
We have to show that $Z_h\ge Z$.
To this end, we may assume that $h$ is a general element of $Q: I$
so that $Q:I\subset I_{Z_h}$.
From the exact sequence
\[
0\to \cO_X(Z) \xrightarrow{\times h} \cO_X(Z-Z_h)\to \cC \to 0
\]
obtained from (2.2), 
 we obtain the surjective map
\[
h_1\: H^1(\cO_X(Z))\xrightarrow{\times h} H^1(\cO_X(Z-Z_h)).
\]
We shall show that $h_1$ is trivial. 
Since $H^1(\cO_X(-Z_h))=0$, tensoring the exact sequence \eqref{eq:Kos} with $\cO_X(-Z_h)$ we obtain the exact sequence
$$
0\to QI_{Z_h} \to I_{Z+Z_h} \to H^1(\cO_X(Z-Z_h))\to 0.
$$
Therefore we may regard $h_1$ as a homomorphism
$$
 h_1\: I/Q \xrightarrow{\times h} I_{Z+Z_h}/QI_{Z_h}.
$$
However it follows from \lemref{l:core} that
$hI\subset (Q:I)Q \subset I_{Z_h}Q$.
Hence the map $h_1$ should be trivial.
By 
Lemma 5.5, we obtain that $Z-Z_h\le 0$.
\end{proof}

%
%
%
%

\section{Number of minimal generators of integrally closed ideals.}

The aim of this section is to study the number
of minimal set of generators for integrally closed ideals in $A$.
In what follows, let $M$ denote the maximal ideal cycle of
a given resolution of singularities $f \colon X \to \Spec A$;
see \cite[Definition 2.11]{yau.max}.
\par
Furthermore, we always assume that an integrally closed $\m$-primary
ideal $I=I_Z$ is represented by $Z$.

\begin{thm} 
\label{mu(I)}
Let $(A,\m)$ and $f \colon X \to \Spec A$ be as above.
Let $I=I_Z$ be an integrally closed $\m$-primary ideal and 
assume that $\frm \cO_X=\cO_X(-M)$.
Then we have an inequality
\[
- MZ +1 \ge \mu_A(I) \ge -MZ + 1 -p_g(A).
\]
More precisely, we have 
\[
\ell_A( I/\overline{I\frm}) = - MZ + 1 -
\varepsilon(Z,M)
\]
and we have equality $\mu_A(I) = -MZ +1$
if 
 $\overline{I\frm} = fI + g\frm$ for general elements $f\in \frm$ and $g\in I$.
\end{thm}

\begin{proof}
In the proof, we write $h^1(-Z)$ instead of $h^1(\mathcal{O}_X(-Z))$
for any anti-nef cycle $Z$ on $X$.
Note that $\mu_A(I) = \ell_A(I/I \m) \ge \ell_A(I/ \overline{I \m})$ and
$\overline{I \m}= I_{M+Z}$.
By Riemann-Roch formula \ref{t:kato}, we have
\begin{eqnarray*}
\ell_A(A/ \overline{I \m}) & = & p_g(A)-h^1(-M-Z) - \frac{(M+Z)^2
+K_X(M+Z)}{2}, \\[2mm]
\ell_A(A/I) & = & p_g(A) - h^1(-Z) - \frac{Z^2 +K_XZ}{2}, \\[2mm]
1=\ell_A(A/ \m) & = & p_g(A) - h^1(-M) -
\frac{M^2 +K_XM}{2}.
\end{eqnarray*}
It follows that
\begin{eqnarray*}
\ell_A(I/ \overline{I \m})
&=& \ell_A(A/\overline{I \m}) - \ell_A(A/I) - \ell_A(A/\m)+1 \\[2mm]
&=& - MZ + 1 - \left\{p_g(A)- h^1(-M)- h^1(-Z)+h^1(-Z-M) \right\} \\
&=& -MZ+1 - \varepsilon(Z,M).
\end{eqnarray*}
Hence the theorem follows from Proposition \ref{vareps}.
\end{proof}

\begin{ex}
\label{mu(I)=}
If $I_Z$ or $\frm$ is a $p_g$-ideal, then  $\mu_A(I_Z) = -MZ +1$.
Actually, we get the equalities $\overline{I\m}= I \m = f I_Z + g \frm$ by Theorem \ref{t:sg}.
\end{ex}

\begin{cor}
\label{mu(I)_ineq}
If $A$ is a rational singularity and $I$ is an integrally
closed $\frm$-primary ideal, then $\mu_A(I) = -MZ +1$.
\end{cor}

\begin{ex} 
\label{elii-mu} 
Assume that the exceptional set of the minimal resolution of
$\Spec A$ consists of
one curve $E\cong \PP^1$ with $E^2=-r$.
Then $\mu_A(I)$ of integrally closed ideal $I$ is of the form
$\mu_A(I) = nr +1$ for some positive integer $n$.
If $A$ is a simple elliptic singularity of
$e_{0}(\m) = r \ge 3$, then
$\mu_A(I) = nr +1$ or $nr$ for every integrally closed ideal $I$.
Actually, if $I= I_Z$ for some
anti-nef cycle $Z$ on $X$ and if we denote $f : X\to X_0$, where $X_0$ is
the minimal resolution,
then $-MZ = - f_*(Z) E = nr$, 
where $E$ is the unique elliptic curve on $X_0$.
\end{ex}

\begin{rem}
\label{rem:mu}
In any  two-dimensional normal local ring $A$, 
if $I'\subset I$ are integrally closed ideals in $A$,
then $\mu_A(I') \ge \mu_A(I)$ holds true
by \cite[Theorems 3,5]{WJ}.
\end{rem}

\section{Ulrich ideals of minimally elliptic and simple elliptic
singularities.}

Let $(A,\m)$ be a Cohen-Macaulay local ring with infinite residue field,
and let $I$ be an $\m$-primary ideal of $A$ and $Q$ a minimal reduction
of $I$.
Then $I$ is called an \textit{Ulrich ideal} if $I^2=QI$ and $I/I^2$
is $A/I$-free. When $A$ is Gorenstein, $I$ is an Ulrich ideal
if and only if it is a good ideal and $\mu_A(I)=\dim A+1$;
see also subsection 2.5.
Thus in the Gorenstein case, we can regard Ulrich ideals as typical
example of good ideals.
In \cite{GOTWY1,GOTWY2,GOTWY3}, the last two authors classified all
Ulrich ideals
for simple singularities and two-dimensional rational singularities.
So the following problem is natural.

\begin{prob}
\label{Prob-Ul}
Let $A$ be a two-dimensional normal local ring.
Classify all Ulrich ideals of $A$.
\end{prob}

In this section, we will prove non-existence theorem for minimally
elliptic singularities
(namely, Gorenstein rings with $p_g(A)=1$) with high multiplicity
and we will complete solution for Problem \ref{Prob-Ul} in the case of
simple elliptic singularities.
Note that in our case, Ulrich ideal $I$ is a good ideal with $\mu_A(I) =3$.

\par
First, by Lemma \ref{good_lem} and Proposition \ref{t:stable}, we have
the following.

\begin{prop}
\label{pg&good} 
 Let $I$ be a good ideal and assume $I \OO_X
= \OO_X( -Z)$ is invertible
with $h^1(\OO_X( -Z))=\pg(A)$. Then $I$ and $I^2$ are integrally closed.
\end{prop}

\begin{rem}
\label{pg1-sep} 
If $\pg(A)=1$ and $h^1(\OO_X(-Z)) =0$,
then one of the following $2$ cases occur; see (2.3):
\begin{enumerate}
\item $I^2 = QI$.
\item $\ell_A(I^2/QI) =1$ and $I^2$ is integrally closed.
\end{enumerate}
\end{rem}

Let us recall some fundamental facts for minimally elliptic singularities.
Let $Z_f$ be the fundamental cycle and $e=-Z_f^2$.
Then $Z_f=M=Z_{K_X}$ on the minimal resolution, $\cO_X(-M)$ has no fixed
component, $\cO_X(-M)$ is generated (i.e., $\m\cO_X=\cO_X(-M)$) 
if $e\ge 2$. 
Moreover, $\m^n$ are integrally closed for all $n\ge 1$ if $e\ge 3$; 
see Laufer \cite{la.minell}.

\begin{lem} 
\label{Minimallyelliptic} 
Let $(A,\frm)$ be a minimally elliptic singularity of degree $e\ge 3$.
Then $M^2 +K_XM=0$ and $h^1(\OO_X(-M)) =0$.
\end{lem}

\begin{proof}
Let $Q$ be a minimal reduction of $\frm$.
As $A$ is Gorenstein with $e_{0}(\m) \ge 3$, we must have $\m^2 \ne Q\m$.
In particular, $\m$ is not a $p_g$-ideal and
$h^1(\OO_X(-M)) =0$.
Also by Riemann-Roch Theorem \ref{t:kato} for $1=\ell_A(A/\frm)$, we
have $M^2 +K_XM=0$.
\end{proof}

The following lemma plays an essential role.

\begin{lem}
\label{good-ell} 
Let $(A,\frm)$ be a minimally elliptic singularity
 and let $I$ be an $\frm$-primary ideal 
such that $\bar I$ is represented by some cycle $Z$ 
on a resolution $X$ of
$\Spec A$ and assume that $I^2 = QI$.
Then $I$ is a good ideal if and only if one of the following cases occurs:
\begin{enumerate}
\item $h^1(\OO_X(-Z))=1$, $K_X Z =0$ and $I$ is integrally closed.
\item $h^1(\OO_X(-Z))=0$ and $K_X Z = 2(1+ \ell_A(\bar{I}/I))$.
\end{enumerate}
\end{lem}

\begin{proof} Since $A$ is Gorenstein, $I$ is good if and only if
$I^2=QI$ and $e_0(I) = 2\cdot \ell_A(A/I)$.
The result follows from Lemma \ref{good_lem} and
Riemann-Roch Theorem \ref{t:kato}.
\end{proof}

Next, let us discuss how far is an Ulrich ideal from integrally closed.

\begin{lem}
\label{bar(I)/I} 
Let $A$ be a minimally elliptic singularity of degree $e \ge 2$.
If $I$ is an Ulrich ideal of $A$, then $\ell_A(\bar{I}/I) \le 1$.
\end{lem}

\begin{proof} We saw a good ideal $I$ is integrally closed if $\bar{I}$ is a
$p_g$-ideal in Proposition \ref{pg&good}.
Hence we may assume that $\bar{I}=I_Z$ with
$h^1(\cO_X(-Z))=0$ on some resolution $X$.
Now, by \thmref{mu(I)}, $\mu(I_Z) \ge -MZ$,
where $M$ is the maximal ideal cycle on $X$.
Now let $X_0$ be the minimal resolution of $\Spec A$ and $f : X\to X_0$
be the contraction. Then we know
that $M= f^*(M_0)$ and $M_0 = -K_{X_0}$, 
where $M_0$ is the maximal ideal cycle on $X_0$.
Now, write
\[
Z = f^*(f_*(Z)) + Y \quad \text{and} \quad K_X = f^*(K_{X_0}) + L.
\]
Then $Z K_X = f_*(Z) K_{X_0} + YL$ and $YL = ZL \le 0$ since $L\ge 0$
and $Z$ is anti-nef.
Moreover,
$MZ = M_0f_*(Z) = -K_{X_0} f_*(Z) \le -K_X Z$.
Thus we have
\[
3 = \mu_A(I)
\ge \mu_A(\bar{I}) - \ell_A(\bar{I}/I)
\ge - MZ -\ell_A(\bar{I}/I)
\ge K_XZ -\ell_A(\bar{I}/I) = \ell_A(\bar{I}/I) +2
\]
and we get the desired inequality.
\end{proof}

\begin{thm}
\label{Ul-min-ell} 
Let $A$ be a minimally elliptic singularity of degree
$e\ge 2$.
\begin{enumerate}
\item If $e\ge 5$, then $A$ has no Ulrich ideals.
\item If $e=4$ and $I$ is an Ulrich ideal with $\bar{I}= I_Z$, then
$\ell_A(\bar{I}/I) =1$ and $\overline{I}$ is represented on the
minimal resolution $X_0$ and $-MZ=4$, where $M$ is the maximal ideal
cycle on $X$.
\item If $A$ has an Ulrich ideal which is a $p_g$-ideal, then $e\le 2$.
\end{enumerate}
\end{thm}

\begin{proof} Note that $e = -M^2\le -MZ$ 
for any anti-nef cycle $Z$.
If $I$ is Ulrich, then $\mu_A(I)=3$ and putting $\bar{I}=I_Z$,
$\mu_A(I_Z) \le 4$. 
On the other hand, $\mu_A(I_Z) \ge -MZ \ge e$ and we
have $e\le 4$. 
If $I_Z$ is an Ulrich $p_g$-ideal,
then by Example \ref{mu(I)=}, we have $3 = \mu_A(I) = - MZ + 1\ge e+1$.
\par
If $e=4$ and $I$ is an Ulrich ideal with $\bar{I}=I_Z$, 
then $\ell_A(\bar{I}/I)=1$ and $-MZ=4$ 
since $\mu_A(I_Z)\ge -MZ\ge e$.
Moreover, by Lemma \ref{good-ell}, we have $K_XZ = 4$.
Using the notation of Lemma \ref{bar(I)/I}, we have $ZL=0$.
Therefore $\bar{I}$ can be represented on the minimal resolution.
\end{proof}

\begin{rem} 
\label{rem_minimally4} 
Let $A$ be a minimally elliptic singularity of
degree $e=4$.
Then it is known that $A$ is a complete intersection
of codimension $2$.
Now, let $I$ be an $\frm$ primary ideal generated
by $3$ elements among the minimal generating system of $\frm$ containing
some minimal reduction $Q$ of $\frm$.
Then we can show that $I$ is an Ulrich ideal.
Let $E= f^{-1}(\frm)$, where $f: X \to \Spec A$ is the blowing up of the
maximal ideal.
Now, consider the family of ideals $I$ generated by $3$ elements among
the minimal generating system of $\frm$.
Then the condition $I$ contains a minimal reduction of $\frm$ is
equivalent to say that the generators of $I$ have no common zero on $E$
as sections of a line bundle $\m\cO_X$.
Hence the family of such ideals forms an open subset of $\PP^3 =
\PP(\frm/\frm^2)$.
If $I$ is an Ulrich ideal, then $\bar{I}=\frm$ by Theorem \ref{Ul-min-ell}.
\end{rem}

\begin{ex} 
\label{minimally4} 
Let $A = k[[x,y,z,w]]/(y^2-xz, w^3 - y(z-x))$ be a complete intersection,
which is a minimally elliptic singularity of degree $e=4$.
Then the minimal
resolution of $\Spec A$ is a star-shaped graph with central curve
$E_0\cong \PP^1$
with $E_0^2=-2$ and $4$ branches $E_i$ ($i=1,2,3,4$) with $E_i^2=-3$.
Here, we have
$M= 2 E_0 + \sum_{i=1}^4 E_i$.
Let $Z = 3 E_0 + \sum_{i=1}^4 E_i$. Then $I_Z = (x,y,z,w^2)$
and $I= (x,y,z)$ is an Ulrich ideal with $\bar{I}=I_Z$.
\end{ex}

\par
In the following, let $(A,\m)$ be a simple elliptic singularity of
degree $e$ (see Section 2).
Let us classify all of the Ulrich ideals of those rings. 
Also, let $I$ be an $\frm$-primary ideal and
$Q$ its minimal reduction.
We assume $I \OO_X = \OO_X( -Z)$ for some
resolution $X$ so that
$\bar{I} = I_Z= \rmH^0(X, \OO_X( - Z))$ is
the integral closure of $I$.
\par
Now we state our classification.

\begin{thm}
\label{Main-ell} 
Let $(A,\frm)$ be a simple elliptic singularity of
degree $e$. Then$:$
\begin{enumerate}
\item If $e\ge 5$, then $A$ has no Ulrich ideals.
\item If $e=4$ and $I$ is an Ulrich ideal,
then $\ell_A(A/I)=2$ and $\bar{I}=\frm$.
\item If $e=3$ and $I$ is an Ulrich ideal,
then $\ell_A(A/I)=2$ and $I$ is integrally closed.
Such ideals consist a family parametrized
by the elliptic curve $E_0$.
\item If $e=2$, then an Ulrich ideal $I$ is one of the followings$;$
\begin{enumerate}
 \item $\m$. 
 \item $I=\bar{I}$ and $\ell_A(A/I)=2$.
 \item $I=\bar{I}$ and $\ell_A(A/I)=3$.
 \item $\bar{I}=\overline{\frm^2}$ and $\ell_A(A/I)=4$.
\end{enumerate}
There are $4$ ideals of type $($c$)$, and the ideals of type $($b$)$ is
parametrized by $\PP^1\setminus \{4$ points$\}$.
\item If $e=1$, then $A$ has integrally closed Ulrich ideals with
$\ell_A(A/I) =1,2,3,4$.
The ones with $\ell_A(A/I)=2$ (resp. $\ell_A(A/I)=3$) are parametrized
by $E_0$
$($resp. $\PP^1\setminus \{3$ points$\}$$)$ and there are exactly $3$ Ulrich
ideals
with $\ell_A(A/I) =4$.
\end{enumerate}
\end{thm}

Incidentally, $A$ is not a hypersurface or complete intersection if and
only if $e\ge 5$. This raises the following Question.

\begin{quest} 
\label{ci-Ulrich} 
Let $(A,\frm)$ be Gorenstein normal local ring of dimension $\ge 2$.
Up to now, if $A$ has an Ulrich ideal, then $A$ is a complete intersection.
Is $A$ a complete intersection if it has an Ulrich ideal?
\end{quest}

\begin{proof}[Proof of Theorem \ref{Main-ell}]
The cases with $e\ge 4$ are treated already in Theorem \ref{Ul-min-ell}. 
So we may assume that $e \le 3$. 
In the following, let $I$ be an Ulrich ideal
of $A$ with $\bar{I}=I_Z$
for some anti-nef cycle $Z$ on some resolution $X \to \Spec A$. 
Let $X_0$ be the minimal resolution with exceptional set $E_0$ 
and let $f \colon X \to X_0$ be the contraction.
Suppose that $f_{*}(Z) = nE_0$.  
\par
(3) If $e=3$, and $I=I_Z$ is a $p_g$-ideal, then by Example \ref{mu(I)=},
$\mu_A(I) \ge 4$. Hence, we may assume $I_Z$ is not a $p_g$-ideal.
If $n\ge 2$, then $\mu_A(I_Z)\ge 6$ by Theorem \ref{mu(I)} 
and thus by Lemma \ref{bar(I)/I}, $\mu_A(I)\ge 5$.
Hence $n=1$ and $K_XZ \le K_{X_0}E_0 =3$. By Lemma \ref{good-ell},
we have $K_XZ =2$ and $I$ is integrally closed.
Since $K_XZ =2$, 
we need exactly one blowing-up from $X_0$ and hence we may assume that 
$f$ is the blowing up of a point $P$ on $E_0$ and $Z= f^*(E_0) + E$, 
where $E=f^{-1}(P)$.
In this case, $\ell_A(A/ I_Z)=2$. Such $I$ is determined by $P\in E_0$.
\par
Now let us show that such $I=I_Z$ is actually an Ulrich ideal. Since
$\ell_A(A/I) =2$
and $e_0(I) = -Z^2=4$, we have only to show that $I^2= QI$ for a reduction
$Q$ of $I$. 
For that purpose, we need to show that $I^2$ is \textit{not} integrally closed;
see Remark \ref{pg1-sep}.
\par
Now, by Riemann-Roch Theorem \ref{t:kato}, we get
$\ell_A(A/\overline{I^2})=7$.
On the other hand, since $I$ is generated by $\frm^2$ and $2$ linear forms,
$\ell_A( A/(I^2 +\frm^3))=7$. 
Hence if $I^2= \overline{I^2}$, then $I^2
+\frm^3 = I^2$,
or $I^2\supset \frm^3$. This contradicts the fact $3 f^*(E_0) \not\ge 2Z$.
\par
(4) Next, assume $e=2$. If $\bar{I}=I_Z$ is a $p_g$-ideal,
then by Lemma \ref{good-ell}, we have $I_Z=\bar{I}=I$ and $K_XZ=0$. 
Also $\mu_A(I) \ge 2n+1$
and hence $n=1$. Then the cases (b), (c) of the theorem occur.
Actually, we know that $\hat{A} \cong k[[x,y,z]]/(x^2 - \phi(y,z))$, where
$\phi(y,z)$ is a homogeneous polynomial of degree $4$ in $(y,z)$.
Take any linear form $l \in I$. 
We may assume that $l$ is a form of $y$ and $z$. 
If $l$ is not a factor of $\phi$, then
the line $l=0$ intersects with $E_0$ in $2$ points $P_1,P_2$.
Let $f: X\to X_0$ be the blowing up of these $2$ points, 
let $E_i =f^{-1}(P_i)$.
and $Z = f^*(E_0) + E_1+E_2$, then we get the case (b). 
If $l$ is a factor of $\phi$,
then $l=0$ intersects $E_0$ at a point $P$ with multiplicity $2$.
Let $f: X\to X_0$ be the blowing up of $P$, let $E_i = f^{-1}(P)$
and $Z = f^{*}(E_0) + 2E$, then we get the case (c). The ideal is
$I = (l, x, (y,z)^2)$ in case (b) and $I = (l, x, (y,z)^3)$ in case (c).
\par
Next assume that $\bar{I} = I_Z$ is not a $p_g$-ideal. 
By Lemma \ref{good-ell},
we have $K_X Z = 2 (1+ \ell_A(\bar{I}/I))$. 
Also, Theorem \ref{mu(I)} and Lemma \ref{bar(I)/I} imply that
\[
K_XZ \le 2n=-MZ \le \mu_A(I_Z) \le \mu_A(I)+\ell_A(\bar{I}/I) \le 4.
\]
Hence $n\le 2$ and $K_X Z\le 2n$. 
It turns out that we have $K_X Z= 2n$ in case $n=1,2$
and we get the cases (a),(d).
\par
In the case $A = k[[x,y,z]]/(x^2 + y^4+z^4)$, the Ulrich ideals are
calculated in Example \ref{simple2}
using the theory of simple singularities.
\par
(5) Finally let us treat the case $e=1$.
In this case, $\frm \cO_{X_0}$ has a base point and
let $X_1$ be the blowing up of the base point. 
We choose $X_1$ as starting point.
The exceptional set of $X_1$ is $E_0\cup E_1$, where $E_0$ is an
elliptic curve with $E_0^2=-2, E_1\cong\PP^1$ with $E_1^2=-1$ 
and $E_0E_1=1$. 
Here, $\frm$ is defined by $M = E_0+2E_1$. 
Note that $K_{X_1}M =0$ and $\frm$ is a $p_g$-ideal.
\par
Let $I$ be an Ulrich ideal of $A$ with $\bar{I}=I_Z$, $f : X \to X_1$ be
contraction and we put $f_*(Z) = aE_0+ bE_1$.
Hence $-Mf_{*}(Z) = b$ and $K_{X_1}f_*(Z) = 2a-b$.
Since $f_{*}(Z)$ is anti-nef, $a\le b$. 

\par
First assume that $I_Z$ is a $p_g$-ideal. 
Then from $\mu(I) =3$, we get $b=2$.
Here if $a=1$, $K_{X_1}f_*(Z)=0$ and $Z=M, I=\frm$. 
If $a=2$, we get the ones with $\ell_A(A/I)= 3,4$.
\par
Actually, we can assume
$\hat{A} \cong k[[x,y,z]]/(x^2 - \phi(y,z^2))$, where
$\phi$ is a homogeneous polynomial of degree $3$ with no multiple roots.
Since $a\ge 2$, $I$ is contained in $(x,y,z^2)$. 
Take any linear form $l \in I$. 
We may assume that $l$ is a form of $y$ and $z^2$. 
If $l$ is not a factor of $\phi$, then
$l=0$ defines $2$ points $P_1,P_2$ on $E_0$.
Let $f: X\to X_1$ be the blowing up  
of these $2$ points 
and put $F_i = f^{-1}(P_i)$ ($i=1,2$).
Then putting $Z= f^*(2E_0+2E_1) + F_1+F_2$, we get $K_XZ=0$ and $I_Z$ is
an Ulrich ideal of $\ell_A(A/I_Z)=3$.
If $l$ is one of $3$ factors of $\phi$, then
$l=0$ intersects $E_0$ at 
a point $P$ 
with multiplicity $2$.
Let $f: X\to X_1$ be the blowing up of $P$ and
put $F= f^{-1}(P)$ .
Then putting $Z= f^*(2E_0+2E_1) + 2F$, we get $K_XZ=0$
and $I_Z$ is an Ulrich ideal of $\ell_A(A/I_Z)=4$.

Next assume that $h^1(-Z)=0$ with
$f_*(Z) = aE_0+ bE_1$.
Then by Example \ref{mu(I)=}, 
$\ell_A(\bar{I}/ \overline{\frm I}) = - MZ+1$.
Hence 
\begin{eqnarray*}
3 =\mu_A(I)& \ge &  - MZ +1 - \ell_A(\bar{I}/I) \\
&=& - M_{X_1}f_*(Z) +1 -\ell_A(\bar{I}/I) \\
&=& (K_{X_1} - 2E_1)f_*(Z) +1- \ell_A(\bar{I}/I) \\
&\ge& K_XZ +1 - \ell_A(\bar{I}/I) \\
&=& 2(1+\ell_A(\bar{I}/I)) +1 - \ell_A(\bar{I}/I) \\
&=& 3 + \ell_A(\bar{I}/I).
\end{eqnarray*} 
Thus we have $\bar{I}=I$.
Then by Lemma \ref{good-ell}, we must have $K_X Z =2$.
On the other hand, $2=K_XZ \le K_{X_1}f_*(Z) = 2a-b$
and $3 = \mu_A(I) \ge - MZ +1 = b +1$.
Hence we must have $a=b=2$, $I = (x,y,z^2)$.
\par
In the case $A = k[[x,y,z]]/(x^2 + y^3+z^6)$,
the Ulrich ideals are calculated in Example \ref{simple1}
using the theory of simple singularities.
\end{proof}

\begin{ex} 
\label{simple3}
Let $(A,\frm)$ be the local ring of the vertex of
the cone over smooth cubic curve $E_0 \subset \PP_k^2$.
Then $A$ is a simple elliptic singularity of degree $e=3$.
The minimal resolution $X_0$ of $A$ is obtained by blowing-up
of the maximal ideal and the exceptional set is $E_0$ with $E_0^2=-3$.
Take a line $l = 0$ in $\PP_k^2$ intersecting with $E_0$ at 3 distinct
points
$P_1,P_2,P_3$.
Let $\pi : X\to X_0$ be obtained by blowing-up these 3 points.
We denote $E_i=\pi^{-1}(P_i)$ the corresponding exceptional curve
($i=1,2,3$)
and we denote $E_0$ the elliptic curve.
Put $Z = E_0 + 2( E_1+E_2+E_3)$.
Then $\OO_X(-Z)\otimes \OO_{E_0}\cong \OO_{E_0} $ and actually,
we get $h^1(\OO_X(-Z))=1$.
If we put $I = \rmH^0( X, \OO_X(-Z))$, then
$I $ is generated by $\frm^2$ and the linear
form $l$ and $I$ is a good ideal
with $e(I) =6$ and $\ell_A(A/I)=3$.
Since $\mu_A(I) = 4$, $I$ is not an Ulrich ideal.
\end{ex}

\begin{ex} 
\label{simple1} 
Let $A=k[[x,y,z]]/(x^2+y^3+z^6)$ be a hypersurface,
which is a simple elliptic singularity of degree $e=1$.
Then
\begin{enumerate}
\item $\m=(x,y,z)$ is an Ulrich ideal of colength $1$ with minimal
reduction $Q = (y,z)$.
\item $I=(x,y,z^2)$ is an Ulrich ideal of colength $2$ with minimal
reduction $Q = (y,z^2)$.
\item For any $\varepsilon \in \mathbb{C}$,
$(x,y+\varepsilon z^2,z^3)$ is an Ulrich ideal of colength $3$ with
minimal reduction $Q = (y+\varepsilon z^2,z^3)$.
\item For an $\varepsilon \in \mathbb{C}^{\times}$, $(x,y+\varepsilon
z^2,z^4)$ is an Ulrich ideal
of colength $4$ if and only if $\varepsilon^3=1$.
Then $Q=(y+\varepsilon z^2,z^4)$ gives a minimal reduction.
\end{enumerate}
\end{ex}

\begin{ex} 
\label{simple2} 
Let $A=k[[x,y,z]]/(x^2+y^4+z^4)$ be a hypersurface, which is a simple
elliptic singularity of degree $e=2$.
\begin{enumerate}
\item $\m=(x,y,z)$ is an Ulrich ideal of colength $1$ with minimal
reduction $Q=(y,z)$.
\item If we put $I=(x,y+\varepsilon z, z^2)$ for some $\varepsilon \in
\mathbb{C}$, then
$I$ is an Ulrich ideal of colength $2$ with minimal reduction
$Q=(y+\varepsilon z, z^2)$.
\item If we put $I=(x,y+\varepsilon z,z^3)$ for some $\varepsilon \in
\mathbb{C}$ with $\varepsilon^4=-1$,
then $I$ is an Ulrich ideal of colength $3$
with minimal reduction $Q=(y+\varepsilon z,z^3)$.
\item $I=(x,y^2,z^2)$, $(x,y^2\pm z^2,yz)$ are Ulrich ideals of
$\ell_A(A/I)=4$ and
$\overline{I}=\overline{\m^2}=(x,y^2,yz,z^2)$.
\end{enumerate}
\end{ex}

\par
We know that any diagonal hypersurface admits an Ulrich ideal if some
exponent is an even number.
How about the case that all exponents are odd numbers?

\begin{ex} 
\label{Diagonal} 
Let $A=k[[x,y,z]]/(x^{2n-1}+y^{2n-1}+z^{2n-1})$
be a hypersurface with $p_g(A)=\genfrac{(}{)}{0pt}{1}{2n-1}{3}$,
where $n \ge 2$ be an integer.
Then $A$ has an Ulrich ideal
$I = (x+y+y^2 +\cdots +y^{n-1},y^{n},z)$.
Actually, if we put $Q=(x+y+y^2 +\cdots +y^{n-1}+y^{n},z)$,
then it is a minimal reduction of $I$ such that
$I^2=QI$ and $\ell_A(A/Q)= 2 \cdot \ell_A(A/I)=2n$.
\par
Note that if $n=2$, then $A$ is a simple elliptic singularity of degree
$e=3$.
\end{ex}

\begin{acknowledgement}  The authors thank Shiro Goto for valuable discussions on good ideals and 
the number of generators of integrally closed ideals. 
\end{acknowledgement}
%

\vspace{-2mm}
\end{document}